\documentclass[11pt, reqno]{amsart}

\author[P.~Leonetti]{Paolo Leonetti}
\address{Department of Statistics, Universit\`a ``L. Bocconi'', via Roentgen 1, 20136 Milan, Italy}
\email{leonetti.paolo@gmail.com}
\author[H.I.~Miller]{Harry I. Miller}
\address{Faculty of Engineering and Natural Sciences, International University of Sarajevo, 71000 Sarajevo, Bosnia-Herzegovina}
\email{himiller@hotmail.com}

\author[L.~Miller-Van Wieren]{Leila Miller-Van Wieren}
\address{Faculty of Engineering and Natural Sciences, International University of Sarajevo, 71000 Sarajevo, Bosnia-Herzegovina}
\email{lejla.miller@yahoo.com}


\title{Duality between Measure and Category \\ of Almost All Subsequences of a Given Sequence}

\usepackage{amsmath}
\usepackage{amssymb}
\usepackage{amsthm}
\usepackage[left=3.41cm, right=3.41cm, paperheight=11.8in]{geometry}
\usepackage{hyperref}
\usepackage{fancyhdr}
\usepackage{enumitem}
\usepackage{comment}
\usepackage{nicefrac}
\usepackage{mathrsfs}
\usepackage{graphicx}
\usepackage[utf8]{inputenc}

\AtBeginDocument{%
   \def\MR#1{}
}

\newtheorem{thm}{Theorem}[section]
\newtheorem{lem}[thm]{Lemma}

\theoremstyle{definition} 
\newtheorem{defi}[thm]{Definition}%[section]
\let\olddefi\defi
\renewcommand{\defi}{\olddefi\normalfont}
\newtheorem{example}[thm]{Example}
\let\oldexample\example
\renewcommand{\example}{\oldexample\normalfont}
\newtheorem{rmk}[thm]{Remark}
\let\oldrmk\rmk
\renewcommand{\rmk}{\oldrmk\normalfont}

\newcommand{\clusterfin}{\mathrm{L}_x}

\hypersetup{
    pdftitle={Subsequences preserving statistical cluster points},
    pdfauthor={Paolo Leonetti},
    pdfmenubar=false,
    pdffitwindow=true,
    pdfstartview=FitH,
    colorlinks=true,
    linkcolor=blue,
    citecolor=green,
    urlcolor=cyan
}

\providecommand{\MR}[1]{}
\providecommand{\bysame}{\leavevmode\hbox to3em{\hrulefill}\thinspace}
\providecommand{\MR}{\relax\ifhmode\unskip\space\fi MR }

\providecommand{\href}[2]{#2}
                                 
\begin{document}

\maketitle
\thispagestyle{empty}

\begin{abstract}
\noindent 
%A real $\ell$ is said to be a statistical cluster point of a real sequence $(x_n)$ provided that the set $\{n: |x_n-\ell|\le \varepsilon\}$ does not have zero asymptotic density for every $\varepsilon>0$. Accordingly, let
%Let $S$ be the set of subsequences $(x_{n_k})$ of a given real sequence $(x_n)$ such that the set of statistical cluster points of $(x_{n_k})$ is equal to the set of statistical cluster points of $(x_n)$. 
Let $S$ be the set of subsequences $(x_{n_k})$ of a given real sequence $(x_n)$ which preserve the set of statistical cluster points.  
It has been recently shown that $S$ is a set of full (Lebesgue) measure. Here, on the other hand, we prove that $S$ is meager if and only if 
%$S^c$ is meager which is, in turn, equivalent to the fact that 
there exists an ordinary limit point of $(x_n)$ which is not also a statistical cluster point of $(x_n)$. This provides a non-analogue between measure and category.
\end{abstract}

\section{Introduction}\label{sec:intro}

Oxtoby's classical book \cite{MR584443} examines analogues and non-analogues of statements about measure and category. To make some examples, recall that every number $\omega \in (0,1]$ has a unique binary representation 
\begin{equation}\label{eq:binaryexpansion}
\omega=\sum_{n\ge 1}\frac{d_n(\omega)}{2^n}
\end{equation}
such that $d_n(\omega)=1$ for infinitely many positive integers $n$. Then, it is well known that the set of normal numbers 
$$
\mathcal{N}:=\left\{\omega \in (0,1]: \lim_{n\to \infty} \frac{1}{n}\sum_{k=1}^n d_k(\omega)=\frac{1}{2}\right\}
$$
has full Lebesgue measure, that is, $\lambda(\mathcal{N})=1$, where $\lambda$ stands for the (completion of the) Lebesgue measure on $\mathbf{R}$. On the other hand, $\mathcal{N}$ is a first category set. Hence $\mathcal{N}$ is ``big'' from a measure theoretic viewpoint, but ``small'' in the topological sense. This gives a non-analogue between measure and category.

In a different direction, let $A, B\subseteq \mathbf{R}$ be two sets with positive inner Lebesgue measure, i.e., 
they contain a closed set with positive Lebesgue measure. 
%$$
%\sup_{F \subseteq A,\, F \text{ closed}}\lambda(F)>0\,\,\,\text{and}\,\,\,\sup_{F \subseteq B,\, F \text{ closed}}\lambda(F)>0.
%$$
Then, by a famous result of Steinhaus, the sumset $A+B:=\{a+b: a \in A, b \in B\}$ has non-empty interior (in particular, it is `big'' in the measure sense), cf. e.g. \cite[Theorem 3.7.1]{MR2467621}. As a topogical analogue, if $A,B \subseteq \mathbf{R}$ are two second category sets with the Baire property, then the sumset $A+B$ has non-empty interior as well (in particular, it is `big'' also in the category sense), cf. \cite[Theorem 2.9.1]{MR2467621}.

The aim of this note is to provide another example of non-analogue between measure and category related to statistical cluster points. Here, we recall that a real $\ell$ is said to be a \emph{statistical cluster point} of a real sequence $x=(x_n)$ provided that 
$$
\forall \varepsilon>0,\,\,\,\,\,\mathrm{d}^\star\left(\{n \in \mathbf{N}: |x_n-\ell|\le \varepsilon\}\right)>0,
$$
where $\mathrm{d}^\star$ stands for the upper asymptotic density, i.e., the function
\begin{equation}\label{eq:dstar}
\mathrm{d}^\star: \mathcal{P}(\mathbf{N}) \to \mathbf{R}: X \mapsto \limsup_{n\to \infty} \frac{|X \cap [1,n]|}{n}.
\end{equation}
Hereafter, we denote the set of statistical cluster points of a real sequence $x$ by $\Gamma_x$ and the set of ordinary limit points by $\clusterfin$ (clearly, $\Gamma_x \subseteq \clusterfin$). 

Statistical cluster points were introduced by Fridy \cite{MR1181163} and then studied by many authors, see e.g. \cite{PaoloMarek17, MR1416085, MR1838788, Leo17, MR1260176, Miller18, MR1924673}. However, this notion has been studied much before under a different name. Indeed, as it follows by \cite[Theorem 4.2]{LMxyz}, statistical cluster points of a real sequence $x$ correspond to classical ``cluster points'' of a filter $\mathscr{F}$ on $\mathbf{R}$ (depending on $x$), cf. \cite[Definition 2, p.69]{MR1726779}.

\section{Main results}

For each $\omega \in (0,1]$ and real sequence $x$, let $x\upharpoonright \omega$ be the subsequence of $(x_n)$ obtained by choosing all the indexes $n$ such that $d_n(\omega)=1$ in the representation \eqref{eq:binaryexpansion},  cf. \cite[Appendix A31]{MR1324786} and \cite{MR1260176}. Accordingly, the following result has been recently shown by the authors, see \cite{Miller18} and \cite[Theorem 3.1]{Leo17}:
\begin{thm}\label{thm:oldmeasure}
Let $x$ be a real sequence. Then 
$\lambda\left(\left\{\omega \in (0,1]: \Gamma_x=\Gamma_{x\upharpoonright \omega}\right\}\right)=1$.
\end{thm}
In other words, almost all subsequences of $x$ preserve the set of statistical cluster points, from a measure theoretic viewpoint. 

Our main result, 
which is proved in Section \ref{sec:proof}, 
provides the topological counterpart of Theorem \ref{thm:oldmeasure}:
\begin{thm}\label{thm:main}
Let $x$ be a real sequence. Then $\left\{\omega \in (0,1]: \Gamma_x=\Gamma_{x\upharpoonright \omega}\right\}$ is not meager if and only if every ordinary limit point is also a statistical cluster point, i.e., $\Gamma_x=\clusterfin$. Moreover, in this case, it is comeager.
\end{thm}
%The proof follows in Section \ref{sec:proof}.

Lastly, given a real sequence $x$,  we recall that a real $\ell$ is said to be a \emph{statistical limit point} of $x$ provided that there exists a subsequence $(x_{n_k})$ converging (in the ordinary sense) to $\ell$ and $\mathrm{d}^\star(\{n_k: k \in \mathbf{N}\})>0$. We denote by $\Lambda_x$ the set of statistical limit points of $x$. 

It is known that the analogue of Theorem \ref{thm:oldmeasure} holds also for statistical limit points, see \cite[Theorem 3.3]{Miller18} and \cite[Theorem 4.2]{Leo17a}. In the same direction, we have the analogue of Theorem \ref{thm:main}:
\begin{thm}\label{thm:main2}
Let $x$ be a real sequence. Then $\left\{\omega \in (0,1]: \Lambda_x=\Lambda_{x\upharpoonright \omega}\right\}$ is not meager if and only if $\Lambda_x=\clusterfin$. Moreover, in this case, it is comeager.
\end{thm}

Proofs follow in Section \ref{sec:proof}. We leave as an open question for the interested reader to check whether Theorem \ref{thm:main} can be extended to every ideal on $\mathbf{N}$. To be precise, let $\mathcal{I}\subseteq \mathcal{P}(\mathbf{N})$ be a collection of subsets closed under taking subsets and finite unions, containing all the finite sets, and different from $\mathcal{P}(\mathbf{N})$ itself. Moreover, given a real sequence $x$, we let $\Gamma_x(\mathcal{I})$ be the set of $\mathcal{I}$\emph{-cluster points} of $x$, that is, the set of all $\ell$ such that $\{n: |x_n-\ell|\le \varepsilon\} \notin \mathcal{I}$ for all $\varepsilon>0$, cf. \cite{LMxyz}. (Note that, if $\mathcal{I}$ is the ideal of zero asymptotic density sets, i.e., 
$$
\mathcal{I}_0:=\left\{X\subseteq \mathbf{N}: \mathrm{d}^\star(X)=0\right\},
$$
then $\Gamma_x=\Gamma_x(\mathcal{I}_0)$.) Accordingly, is it true that $\left\{\omega \in (0,1]: \Gamma_x(\mathcal{I})=\Gamma_{x\upharpoonright \omega}(\mathcal{I})\right\}$ is not meager if and only if $\Gamma_x(\mathcal{I})=\clusterfin$?

\section{Proof of Theorems \ref{thm:main} and \ref{thm:main2}}\label{sec:proof}
We start with a preliminary lemma.
\begin{lem}\label{lem:prelim}
Let $x$ be a real sequence with an ordinary limit point $\ell$. Then 
$$
S(\ell):=\left\{\omega \in (0,1]: \ell \in \Gamma_{x\upharpoonright \omega}\right\}
$$
is comeager.
\end{lem}
\begin{proof}
Fix $\omega \in (0,1]$ and note that $\ell$ is a statistical cluster point of the subsequence $x\upharpoonright \omega$ if and only if
$$
\forall m \in \mathbf{N},\,\,\,\,\,\mathrm{d}^\star\left(\left\{n \in \mathbf{N}: |\,(x\upharpoonright \omega)_n-\ell\,|\le \frac{1}{m}\right\}\right)>0,
$$
where $\mathrm{d}^\star$ stands for the upper asymptotic density defined in \eqref{eq:dstar}. Hence, $\ell \in \Gamma_{x\upharpoonright \omega}$ if and only if %there exists $k \in \mathbf{N}$ such that 
$$
\forall m \in \mathbf{N},\exists k \in \mathbf{N},\,\,\,\,\,\, \frac{|\{n\le N: |\,(x\upharpoonright \omega)_n-\ell\,|\le \nicefrac{1}{m}\}|}{N} \ge \frac{1}{2k}
$$
for infinitely many $N \in \mathbf{N}$. Setting 
$$
q_{\omega,m}(N):=\frac{|\{n\le N: |\,(x\upharpoonright \omega)_n-\ell\,|\le \nicefrac{1}{m}\}|}{N}
$$
for each $m,N \in \mathbf{N}$ and $\omega \in (0,1]$, it follows that the set $\left\{\omega: \ell \in \Gamma_{x\upharpoonright \omega}\right\}$ can be rewritten as
$$%\begin{equation}\label{eq:setlemma}
\bigcap_{m \in \mathbf{N}} \bigcup_{k \in \mathbf{N}} \left\{\omega: \exists^\infty N \in \mathbf{N}, q_{\omega,m}(N) \ge \frac{1}{2k}\right\}.
$$%\end{equation}

Accordingly, we have equivalently to prove that the set 
\begin{displaymath}
\begin{split}
\left\{\omega: \ell \notin \Gamma_{x\upharpoonright \omega}\right\} &=\bigcup_{m \in \mathbf{N}} \bigcap_{k \in \mathbf{N}} \left\{\omega: \forall^\infty N \in \mathbf{N}, q_{\omega,m}(N) < \frac{1}{2k}\right\}\\
&=\bigcup_{m \in \mathbf{N}} \bigcap_{k \in \mathbf{N}} \bigcup_{n \in \mathbf{N}} \bigcap_{N \ge n} \left\{\omega: q_{\omega,m}(N) < \frac{1}{2k}\right\}
\end{split}
\end{displaymath}
is a meager set. To this aim, it will be enough to show that the sets
$$
Q(m,k,n):=\bigcap_{N \ge n} \left\{\omega: q_{\omega,m}(N) < \frac{1}{2k}\right\}
$$
are nowhere dense for each $m,k,n \in \mathbf{N}$.

Fix $m,k,n \in \mathbf{N}$. It is well known that $Q(m,k,n)$ is nowhere dense if and only if, for each non-empty interval $I \subseteq (0,1]$, there exists a non-empty interval $J\subseteq I$ such that
\begin{equation}\label{eq:claim}
Q(m,k,n) \cap J=\emptyset.
\end{equation}
Since $\ell$ is an ordinary limit point of $x$, there exists $\omega_0 \in I$ such that 
$$
\forall^\infty N \in \mathbf{N},\,\,\,\,q_{\omega_0,m}(N) \ge \frac{1}{2}.
$$
In particular, for each $n \in \mathbf{N}$, there exists $N_n \in \mathbf{N}$ greater than $n$ such that $q_{\omega_0,m}(N_n) \ge \nicefrac{1}{2}$. Lastly, fix $n \in \mathbf{N}$ and define the non-empty interval
$$
J:=\left\{\omega \in (0,1]: \forall i\le N_n, d_i(\omega)=d_i(\omega_0)\right\}.
$$
It follows by construction that \eqref{eq:claim} holds, completing the proof.
\end{proof}

At this point, we can prove our main result.
\begin{proof}
[Proof of Theorem \ref{thm:main}] Let $x$ be a sequence such that $\Gamma_x=\clusterfin$. First, it is claimed that 
$$
S:=\left\{\omega \in (0,1]: \clusterfin=\Gamma_{x\upharpoonright \omega}\right\}
$$
is a comeager set. This is trivial if $\clusterfin$ is empty. Hence, let us suppose hereafter that $\clusterfin \neq \emptyset$. 

Since $\clusterfin$ is a non-empty closed set, there exists a non-empty countable dense subset $\{\ell_n: n \in \mathbf{N}\}$. Also, since $\Gamma_{x\upharpoonright \omega}$ is closed, we have that
$$
S=\bigcap_{n \in \mathbf{N}}S(\ell_n).
$$
By Lemma \ref{lem:prelim} each $S(\ell_n)$ is comeager, hence $S$ is comeager.

Conversely, assume that $\Gamma_x \subsetneq \clusterfin$. Considering that $S$ is comeager, then also $\{\omega: \Gamma_x \neq \Gamma_{x\upharpoonright \omega}\}$ is comeager, thus $\{\omega: \Gamma_x = \Gamma_{x\upharpoonright \omega}\}$ is meager.
\end{proof}

\begin{proof}
[Proof of Theorem \ref{thm:main2}]
Analogous to the proof of Lemma \ref{lem:prelim}, it can be verified that for each $k \in \mathbf{N}$ and $\ell \in \clusterfin$ the set
\begin{displaymath}
\begin{split}
T_k(\ell)&:=\left\{\omega: \mathrm{d}^\star\left(\{n: |(x\upharpoonright \omega)_n-\ell|\le \varepsilon\}\right) \ge \frac{1}{k} \text{ for all }\varepsilon>0\right\}\\
&\,=\bigcap_{m \in \mathbf{N}} \left\{\omega: \exists^\infty N \in \mathbf{N}, q_{\omega,m}(N) \ge \frac{1}{k}\right\}
\end{split}
\end{displaymath}
is comeager. Moreover, for each $\omega \in (0,1]$ and $k \in \mathbf{N}$ the set
$$
\Lambda_{x\upharpoonright \omega}^k:=\left\{\ell: \mathrm{d}^\star\left(\{n: |(x\upharpoonright \omega)_n-\ell|\le \varepsilon\}\right) \ge \frac{1}{k} \text{ for all }\varepsilon>0\right\}
$$
is closed (note that $\Lambda_{x\upharpoonright \omega}=\bigcup_{k \in \mathbf{N}}\Lambda_{x\upharpoonright \omega}$ is not necessarily closed).

Letting $\{\ell_n: n \in \mathbf{N}\}$ be a countable dense subset of $\clusterfin$, we conclude that
$$
\left\{\omega: \Lambda_{x\upharpoonright \omega}=\clusterfin\right\}=\bigcup_{k \in \mathbf{N}} \bigcap_{n \in \mathbf{N}} T_k(\ell_n) \supseteq \bigcap_{n \in \mathbf{N}} T_1(\ell_n)
$$
is comeager.

The rest is identical to the proof of Theorem \ref{thm:main}.
%The analogue of Lemma \ref{lem:prelim} holds for statistical limit points, once we replace the representation \eqref{eq:setlemma} with
%\begin{displaymath}
%\begin{split}
%\{\omega: \ell \in \Lambda_x\}&=\bigcup_{k \in \mathbf{N}} \left\{\omega: \mathrm{d}^\star\left(\{n: |(x\upharpoonright \omega)_n-\ell|\le \varepsilon\}\right) \ge \frac{1}{k} \text{ for all }\varepsilon>0\right\}\\
%&=\bigcup_{k \in \mathbf{N}} \bigcap_{m \in \mathbf{N}} \left\{\omega: \exists^\infty N \in \mathbf{N}, q_{\omega,m}(N) \ge \frac{1}{k}\right\}
%\end{split}
%\end{displaymath}
%(we leave the details to the reader). 
%
%Then, the proof goes similarly as the one of Theorem \ref{thm:main}.
\end{proof}

Note that it follows by the proofs of Theorem \ref{thm:main} and \ref{thm:main2} that, given a real sequence $x$, the set $\{\omega \in (0,1]: \Lambda_{x\upharpoonright \omega}=\Gamma_{x\upharpoonright \omega}=\clusterfin\}$ is comeager.

\end{document}